\documentclass[11pt]{article}

\usepackage{amsmath, amsfonts,amssymb, amsthm}
\usepackage{graphicx}
\usepackage{color}
\usepackage{epsf}
\usepackage{tikz}
\usepackage{subfig}
\usepackage{hyperref}
\hypersetup{
    colorlinks = true,
    linkcolor = {black},
    citecolor = {black}
}

\usepackage[left=1.3in, right=1.3in, top=1.2in, bottom=1.25in]{geometry}

\usepackage{enumitem}
\setlist{itemsep=0pt, topsep=0pt}

\newcommand{\floor}[1]{\lfloor#1\rfloor}
\newcommand{\ceiling}[1]{\lceil#1\rceil}

\newcommand{\tbf}[1]{\textbf{#1}}

\newtheorem{theorem}{Theorem}[section]

\newtheorem{claim}[theorem]{Claim}

\newtheorem{example}[theorem]{Example}
\newtheorem{problem}[theorem]{Problem}

\newenvironment{proofclaim}[1][Proof of claim]{\begin{proof}[#1]}{\end{proof}}

\newcommand{\mc}{\mathrm{mc}}

\providecommand{\keywords}[1]
{
  \small	
  \textbf{Keywords--} #1
}

\title{Large monochromatic components in hypergraphs with large minimum codegree}
\author{Deepak Bal$^{1}$  \and Louis DeBiasio$^{2}$}
\date{\today}

\begin{document}

\maketitle
\noindent\footnotetext[1]{Department of Mathematics, Montclair State University, Montclair, NJ {\tt deepak.bal@montclair.edu}. }
\noindent\footnotetext[2]{Department of Mathematics, Miami University, Oxford, OH. \texttt{debiasld@miamioh.edu}. Research supported in part by NSF grant DMS-1954170.
}

\begin{abstract}
A result of Gy\'arf\'as \cite{Gy} says that for every $3$-coloring of the edges of the complete graph $K_n$, there is a monochromatic component of order at least $\frac{n}{2}$, and this is best possible when $4$ divides $n$.  Furthermore, for all $k\geq 3$ and every $(k+1)$-coloring of the edges of the complete $k$-uniform hypergraph $K_n^{k}$, there is a monochromatic component of order at least $\frac{kn}{k+1}$ and this is best possible for all $n$.  

Recently, Guggiari and Scott \cite{GSc} and independently Rahimi \cite{R} proved a strengthening of the graph case in the result above which says that the same conclusion holds if $K_n$ is replaced by any graph on $n$ vertices with minimum degree at least $\frac{5n}{6}-1$; furthermore, this bound on the minimum degree is best possible.

We prove a strengthening of the $k\geq 3$ case in the result above which says that the same conclusion holds if $K_n^k$ is replaced by any $k$-uniform hypergraph on $n$ vertices with minimum $(k-1)$-degree at least $\frac{kn}{k+1}-(k-1)$; furthermore, this bound on the $(k-1)$-degree is best possible.  
\end{abstract}


\section{Introduction}
Given a hypergraph $G$, the \emph{$t$-shadow} of $G$ is the $t$-uniform hypergraph on $V(G)$ with edge set $\{f\in \binom{V(G)}{t}: \exists e\in E(G), f\subseteq e\}$.  We say that a hypergraph $G$ is connected if the 2-shadow of $G$ is connected.  Given a hypergraph $G$ and a positive integer $r$, let $\mc_r(G)$ be the largest integer $t$ such that every $r$-coloring of the edges of $G$ contains a monochromatic component of order at least $t$. Let $K_n^k$ denote the complete $k$-uniform hypergraph on $n$ vertices (and $K_n = K_n^2$ as usual).

The starting point of this paper is the problem of determining the value of $\mc_r(K_n^k)$ for all $n,r,k\geq 2$.  While this problem has been studied since the 70's, it is still open in most cases.  For example, the exact value (even asymptotically) of $\mc_r(K_n^k)$ is not known when $k=2$, $r=7$; $k=3$, $r=8$; or $k=4$, $r=6$ to name a few small cases.  

As for positive results, Gy\'arf\'as proved the following.

\begin{theorem}[Gy\'arf\'as \cite{Gy}]~\label{thm:Gy}
\begin{enumerate}
\item For all $n\geq r\geq 2$, $\mc_r(K_n)\geq \frac{n}{r-1}$.  This is best possible when $(r-1)^2$ divides $n$ and there exists an affine plane of order $r-1$.  
\item For all $n\geq k\geq 2$, $\mc_k(K_n^k)=n$.
\item For all $n\geq k\geq 3$, $\mc_{k+1}(K_n^{k})\geq \frac{kn}{k+1}$.  This is best possible\footnote{Simply partition a set of $n$ vertices as equally as possible into $k+1$ sets $A_1, \dots, A_{k+1}$.  Now consider a complete $k$-uniform hypergraph $K$ with vertex set $A_1\cup \dots \cup A_{k+1}$ where for every edge $e\in E(K)$ we assign it a color $i\in [k+1]$ such that $e\cap A_{i}=\emptyset$ (which must exist since $|e|=k$).  Note that for all $i\in [k+1]$, every component of color $i$ has order at most $n-|A_i|\leq n-\floor{\frac{n}{k+1}}=\ceiling{\frac{kn}{k+1}}$.} for all $n$.
\end{enumerate}
\end{theorem}

We will mention a few other cases in which the value of $\mc_r(K_n^k)$ is known in the conclusion.

A natural question which has received attention lately has been to determine conditions under which a $k$-uniform hypergraph $G$ on $n$ vertices satisfies $\mc_r(G)= \mc_r(K_n^k)$.

Note that for $1\leq r\leq k$, $\mc_r(G)=n= \mc_r(K_n^k)$ if and only if the $r$-shadow of $G$ is complete.  Indeed, if every $r$-set of $G$ is contained an edge, then since $\mc_r(K_n^r)=n$, we have $\mc_r(G)=n$.  Furthermore, if some $r$-set $\{x_1, \dots, x_r\}$ is not contained in an edge, then we can color the edges of $G$ with $r$-colors such that color $i$ is never used on $x_i$ and thus $\mc_r(G)<n$.  

On the other hand, as first noted by Gy\'arf\'as and S\'ark\"ozy \cite{GSc}, when $r>k=2$ it is surprisingly possible for $\mc_r(G)= \mc_r(K_n)$ provided $G$ has large enough minimum degree.  Improving on results in \cite{GS2} and \cite{DKS}, Guggiari and Scott \cite{GSc} and independently Rahimi \cite{R} proved that if $\delta(G)\geq 5n/6-1$, then $\mc_3(G)\geq \frac{n}{2}$. Then F\"uredi and Luo \cite{FL} proved that for $r\geq 4$, if $\delta(G)\geq (1-\frac{1}{6(r-1)^3})n$, then $\mc_r(G)\geq \frac{n}{r-1}$ (and thus $\mc_r(G)= \mc_r(K_n)$ when $(r-1)^2$ divides $n$ and there exists an affine plane of order $r-1$).  It is still an open problem to determine the exact minimum degree threshold which guarantees $\mc_r(G)\geq \frac{n}{r-1}$ for $r\geq 4$, but it is known to be at least $(1-\frac{r-2}{r^2-r})n-1$ (see \cite{DK} and \cite{GSc}).  

The purpose of the paper is to consider a generalization of the results mentioned in the previous paragraph to $k$-uniform hypergraphs with $k\geq 3$.  Given a hypergraph $G$ and a set $S\subseteq V(G)$, we let $\deg(S)$ be the number of edges of $G$ which contain $S$.  Given $0\leq \ell\leq k$ and a $k$-uniform hypergraph $G$, we let $\delta_{\ell}(G)=\min_{S\in \binom{V(G)}{\ell}}\deg(S).$  When $\ell=k-1$ it is customary to refer to $\delta_{k-1}(G)$ as the minimum \emph{codegree} of $G$.     

Our main result is a strengthening of Theorem \ref{thm:Gy}(iii) for $k$-uniform hypergraphs with sufficiently large $(k-1)$-degree.

\begin{theorem}\label{thm:mindeg}
For all $n\geq k\geq 3$, if $G$ is a $k$-uniform hypergraph on $n$ vertices with 
\begin{equation}\label{codeg}
\delta_{k-1}(G)\geq \frac{kn}{k+1}-(k-1), 
\end{equation}
then $\mc_{k+1}(G)\geq \frac{kn}{k+1}$.  Furthermore, the degree condition is best possible when $n\geq k(k+1)$.  
\end{theorem}

\section{Large minimum codegree implies large monochromatic components}\label{sec:main}

We begin by giving the example which shows that the degree condition in Theorem \ref{thm:mindeg} is best possible.  

\begin{example}
For all $k\geq 3$ and $n\geq k(k+1)$, there exists an $k$-uniform hypergraph $G$ on $n$ vertices with $\delta_{k-1}(G)= \ceiling{\frac{kn}{k+1}}-k$ for which $\mc_{k+1}(G)\leq \mc_{k}(G)\leq  \ceiling{\frac{kn}{k+1}}-1$.
\end{example}

\begin{proof}
Let $A_1, \dots, A_k, A_{k+1}$ be disjoint sets such that $|A_i|=\floor{\frac{n}{k+1}}+1$ for all $i\in [k]$ and $|A_{k+1}|=n-k(\floor{\frac{n}{k+1}}+1)$.  Note that since $n\geq k(k+1)$, we have $|A_{k+1}|\geq 0$.  

We create an $k$-uniform hypergraph $G$ on $V=A_1\cup \dots\cup A_{k+1}$ by adding all possible edges except those which touch all of $A_1, \dots, A_k$.  Note that for all $S\in \binom{V}{k-1}$, we either have $\deg(S)=n-(k-1)$ or $\deg(S)=n-(k-1)-(\floor{\frac{n}{k+1}}+1)=\ceiling{\frac{kn}{k+1}}-k$.  

Now for every edge $e$ we color it with some color $i\in [k]$ such that such that $e\cap A_i=\emptyset$ (by the way $G$ is defined, some such $i\in [k]$ must exist).  Thus for all $i\in [k]$, the largest monochromatic component of color $i$ has order at most $n-(\floor{\frac{n}{k+1}}+1)=\ceiling{\frac{kn}{k+1}}-1$.
\end{proof}

Now we move on to the main part of the proof.  

\begin{proof}[Proof of Theorem \ref{thm:mindeg}]
Let $G$ be an $k$-uniform hypergraph on $n$ vertices satisfying \eqref{codeg} and suppose the edges of $G$ have been colored with $k+1$ colors.  Let $x\in V(G)$ and for all $i\in [k+1]$, let $C_i$ be the component of color $i$ which contains $x$.  Note that even if $x$ is not incident with any edges of color $i$, there is still a trivial component of color $i$ (consisting only of $x$) which contains $x$.  Now for all $i\in [k+1]$, let $F_i=V(G)\setminus V(C_i)$.  If $|F_i|\leq \frac{n}{k+1}$ for some $i\in [k+1]$, then we are done, so suppose that $|F_i|>\frac{n}{k+1}$ for all $i\in [k+1]$.  Let $F=F_1\cup \dots \cup F_{k+1}$.

\begin{claim}\label{clm:3way*}
For all distinct $h,i,j\in [k+1]$, $F_h\cap F_i\cap F_j=\emptyset$.
\end{claim}

\begin{proofclaim}
Suppose for contradiction that $F_h\cap F_i\cap F_j\neq \emptyset$ for some distinct $h,i,j\in [k]$ and without loss of generality suppose $\{h,i,j\}=[3]$.  Let $x_{123}\in F_1\cap F_2\cap F_3$ and let $x_i\in F_i$ for all $4\leq i\leq k$.  Note that $x_{123}, x_4, \dots, x_{k}$ are not necessarily distinct, but regardless there is a $k-1$ set $S$ which contains all of $x,x_{123}, x_4, \dots, x_{k}$.  By the choice of $x,x_{123}, x_4, \dots, x_{k}$, we have that every edge containing $S$ must have color $k+1$ and thus $|V(C_{k+1})|\geq k-1+\delta_{k-1}(G)\geq \frac{kn}{k+1}$, contradicting the fact that $|F_{k+1}|>\frac{n}{k+1}$.
\end{proofclaim}

\begin{claim}\label{clm:2way2*}
For all distinct $g,h,i,j\in [k+1]$, $F_g\cap F_h=\emptyset$ or $F_i\cap F_j=\emptyset$.
\end{claim}

\begin{proofclaim}
Suppose for contradiction that $F_g\cap F_h\neq \emptyset$ and $F_i\cap F_j\neq \emptyset$ for some distinct $g,h,i,j\in [k+1]$ and without loss of generality suppose $\{g,h\}=\{1,2\}$ and $\{i,j\}=\{3,4\}$.  Let $x_{12}\in F_1\cap F_2$ and $x_{34}\in F_3\cap F_4$ and let $x_i\in F_i$ for all $5\leq i\leq k$.  Note that $x_{12}, x_{34}, x_5, \dots, x_{k}$ are not necessarily distinct, but regardless there is a $k-1$ set $S$ which contains all of $x,x_{12}, x_{34}, x_5, \dots, x_{k}$.  Note that when $k=3$, we immediately have a contradiction since $x_{34}$ and $x$ cannot be contained together in an edge of color $4$.  When $k\geq 4$, by the choice of $x,x_{12}, x_{34}, x_5 \dots, x_{k}$, we have that every edge containing $S$ must have color $k+1$ and thus $|V(C_{k+1})|\geq k-1+\delta_{k-1}(G)\geq \frac{kn}{k+1}$, contradicting the fact that $|F_{k+1}|>\frac{n}{k+1}$.  
\end{proofclaim}

For all $i\in [k+1]$, let $F_i^*=F_i\setminus (\bigcup_{j\in [k+1]\setminus \{i\}}F_j)$; i.e.~$F_i^*$ is the set of vertices which are in $F_i$ but not in $F_j$ for any $j\in [k+1]\setminus \{i\}$.

\begin{claim}\label{clm:1way2way*}
For all distinct $h,i,j\in [k+1]$, $F_h\cap F_i=\emptyset$ or $F_j^*=\emptyset$.
\end{claim}

\begin{proofclaim}
Suppose for contradiction that $F_h\cap F_i\neq \emptyset$ and $F_j^*\neq \emptyset$ for some distinct $h,i,j\in [k+1]$ and without loss of generality suppose $\{h,i\}=\{1,2\}$ and $j=3$.  Let $x_{12}\in F_1\cap F_2$ and $x_3\in F_3^*$ and let $x_i\in F_i$ for all $4\leq i\leq k$.  Note that $x_{12}, x_{3}, \dots, x_{k}$ are not necessarily distinct, but regardless there is a $k-1$ set $S$ which contains all of $x_{12}, x_{3}, \dots, x_{k}$.  By the choice of $x_{12}, x_{3}, x_4 \dots, x_{k}$, we have that $x_{12}\in C_j$ for all $j\in [k+1]\setminus\{1,2\}$ and $x_3\in C_j$ for all $j\in [k+1]\setminus\{3\}$  and thus every edge containing $S$ must have color $k+1$.  Since $x_{12}\in C_{k+1}$, this means that $|V(C_{k+1})|\geq k-1+\delta_{k-1}(G)\geq \frac{kn}{k+1}$, contradicting the fact that $|F_{k+1}|>\frac{n}{k+1}$.
\end{proofclaim}

Since $|F_i|> \frac{n}{k+1}$ for all $i\in [k+1]$, we have (by pigeonhole) that $F_i\cap F_j\neq \emptyset$ for some distinct $i,j\in [k+1]$; without loss of generality, say $i=k$ and $j=k+1$.  Furthermore, by Claim \ref{clm:3way*} we have that $F_{k}\cap F_{k+1}=(F_{k}\cap F_{k+1})\setminus (F_1\cup\dots \cup F_{k-1})$.  Thus by Claim \ref{clm:2way2*}, we have that for all distinct $i,j\in [k-1]$, $F_i\cap F_j=\emptyset$, and by Claim \ref{clm:1way2way*} we have that for all $i\in [k-1]$, $F_i^*=\emptyset$.

So for all $i\in [k-1]$, we have $F_i^*=\emptyset$, $F_i\cap F_{k}\cap F_{k+1}=\emptyset$, and for all $j\in [k-1]\setminus\{i\}$, $F_i\cap F_j=\emptyset$, thus $$|F_i\cap F_{k}|+|F_i\cap F_{k+1}|= |F_i|> \frac{n}{k+1}.$$  Let $i=k-1$ say, and suppose without loss of generality that $|F_{k-1}\cap F_{k+1}|> \frac{n}{2(k+1)}>0$.  Thus by Claim \ref{clm:2way2*} we have that for all $j\in [k-2]$, $F_j\cap F_{k}=\emptyset$ which in turn implies that for all $j\in [k-2]$, $|F_j\cap F_{k+1}|= |F_j|>\frac{n}{k+1}$.  This in turn implies (again by Claim \ref{clm:2way2*}) that $|F_{k-1}\cap F_{k+1}|= |F_{k-1}|>\frac{n}{k+1}$.  Now by Claim \ref{clm:1way2way*}, we have $F_{k}^*=\emptyset$.  Thus $|F_{k}\cap F_{k+1}|= |F_{k}|>\frac{n}{k+1}$.  So all together this implies that we have $k$ disjoint sets, $F_1, \dots, F_{k}$, each of which has more than $\frac{n}{k+1}$ vertices.  

Finally, we note that by \eqref{codeg} and the fact that $|F_i|>\frac{n}{k+1}$ for all $i\in [k]$, the subhypergraph of $G$ consisting of those edges which touch all of $F_1, \dots, F_k$ is connected.  Furthermore, by the definition of the $F_i$'s, every such edge must have color $k+1$ and thus we have a component of color $k+1$ of order at least $|F_1\cup \dots\cup F_k|>\frac{kn}{k+1}$.
\end{proof}

\section{Conclusion}

For all $1\leq \ell\leq k-1$ and $n\geq r\geq k+1\geq 4$, let $f(n,r,k,\ell)$ be the smallest integer for which it holds that if $G$ is a $k$-uniform hypergraph on $n$ vertices with minimum $\ell$-degree at least $f(n,r,k,\ell)$, then $\mc_{r}(G)=\mc_r(K_n^k)$. Thus Theorem \ref{thm:mindeg} determines $f(n,k+1,k,k-1)$ for all $k\ge 3$ and $n\geq k(k+1)$.

While we gave a best possible strengthening of Theorem \ref{thm:Gy}(iii) in terms of minimum  $(k-1)$-degree, it would be interesting to consider minimum $\ell$-degree for $1\leq \ell\leq k-2$; in particular $\ell=1$.

We note that for $r\geq k+2$, determining the (asymptotic) value of $\mc_r(K_n^k)$ itself is mostly open and thus the general problem of determining $f(n,r,k,k-1)$ is likely out of the question at the moment.  That being said, there is one general result of 
F\"uredi and Gy\'arf\'as which for all $k\geq 3$, determines the exact value of $\mc_r(K_n^k)$ for infinitely many $r$ and $n$.

\begin{theorem}[F\"uredi and Gy\'arf\'as~\cite{FG}]\label{fgthm}
Let $k,r\geq 2$ and let $q$ be the smallest integer such that $r\leq q^{k-1}+q^{k-2}+\dots+q+1$.  Then $\mc_r(K^k_n)\geq \frac{n}{q}$. 
This is sharp when $q^k$ divides $n$, $r=q^{k-1}+q^{k-2}+\dots+q+1$, and an affine space of dimension $k$ and order $q$ exists.
\end{theorem}

As for the case $k=3$ and small values of $r$ not covered by any of the results mentioned thus far, Gy\'arf\'as and Haxell proved the following.  

\begin{theorem}[Gy\'arf\'as and Haxell~\cite{GH}]\label{ghthm}
$\mc_{5}(K^3_n)\geq 5n/7$ and $\mc_{6}(K^3_n)\geq 2n/3$.  Furthermore these are tight when $n$ is divisible by $7$ and $6$ respectively.  
\end{theorem}

So we conclude with the following problems which may be within reach.  

\begin{problem}~
\begin{enumerate}
\item Determine the value of $f(n, k+1, k, 1)$ for all $k\geq 3$.  
\item Determine the value of $f(n, r, k, k-1)$ for those $r$ and $k$ mentioned in Theorem \ref{fgthm} and Theorem \ref{ghthm}.  
\end{enumerate}
\end{problem}

\bigskip

\noindent
\tbf{Acknowledgements:}  We thank the referees for their helpful comments.

\end{document}